\documentclass{article}
\usepackage[english]{babel}
\usepackage{amsrefs,amssymb,amsmath,amsthm,amsfonts,epsfig,graphicx,psfrag}
\usepackage{xcolor}
\usepackage[all,cmtip]{xy}
\usepackage[utf8]{inputenc}

\theoremstyle{plain}
\newtheorem{thm}{Theorem}[section]

\newtheorem{lem}[thm]{Lemma}

\theoremstyle{definition}

\theoremstyle{remark}

\DeclareMathOperator{\diam}{diam}

\DeclareMathOperator{\interior}{int}
\DeclareMathOperator{\dist}{dist}

\newcommand{\Z}{\mathbb Z}

\renewcommand{\epsilon}{\varepsilon}

\begin{document}

\author{Alfonso Artigue\footnote{Departamento de Matem\'atica y Estad\'{\i}stica del
Litoral, Universidad de la Rep\'ublica, Gral. Rivera 1350, Salto,
Uruguay. The first author is partially supported  by PEDECIBA,
Uruguay.  Email: aartigue@fing.edu.uy} and Dante
Carrasco-Olivera\footnote{Departmento of Matem\'atica, Universidad
de B\'{\i}o-B\'{\i}o, Casilla 5--C, Concepci\'on, Chile. The second
author is partially supported by project FONDECYT 11121598,
CONICYT-Chile. Email: dcarrasc@ubiobio.cl}}
\title{\textcolor{black}{A Note on Measure-Expansive Diffeomorphisms}}
\date{\today}
\maketitle

\begin{abstract}
\textcolor{black}{In this note we prove that a homeomorphism is
countably-expansive if and only if it is measure-expansive. 
This result is applied for showing that
the \mbox{$C^1$-interior} of the sets
of expansive, measure-expansive and continuum-wise expansive
\mbox{$C^1$-diffeomorphisms} coincide}.
\end{abstract}

\section{Introduction}
\textcolor{black}{The phenomenon of expansiveness occurs when the
trajectories of nearby points are separated by the dynamical
system.}
The first research who considered expansivity in dynamical systems
was by Utz \cite{U50}. There, he defined the notion of unstable
\textcolor{black}{homeomorphism}. An extensive literature related to
properties of expansiveness can be found in
\textcolor{black}{\cites{BW,BW69,CZ09,CS99,F89,G72,GH,H87,JU60,K98,L89,M79,O70,R65,RR87,Sa95,S72,TV99,W66,W69}.}

If $f\colon M\to M$ is a homeomorphism of a compact metric space
$(M,\dist)$ and  if $\delta>0$ we define
\[
 \Gamma_\delta(x)=\{y\in M:\dist(f^n(x),f^n(y))\leq\delta\hbox{ for all }n\in\Z\}.
\]
Let us recall some definitions \textcolor{black}{that can be found
for example in} \cite{MoSi}. We say that $f$ is \emph{expansive} if
there is $\delta>0$ such that $\Gamma_\delta(x)=\{x\}$ for all $x\in
M$. Given \textcolor{black}{a Borel} probability measure $\mu$ on
$M$ we say that $f$ is $\mu$-\emph{expansive} if there is $\delta>0$
such that for all $x\in M$ it holds that $\mu(\Gamma_\delta(x))=0$.
\textcolor{black}{In this case we also say that $\mu$ is an
\emph{expansive measure} for $f$.} We say that $f$ is
\emph{measure-expansive} if it is $\mu$-expansive for every
non-atomic Borel probability measure $\mu$. \textcolor{black}{Recall
that $\mu$ is non-atomic if $\mu(\{x\})=0$ for all $x\in M$.} The
corresponding concepts for flows have been considered
\textcolor{black}{in} \cite{CaMo}. Moreover, we say that $f$ is
\emph{countably-expansive} if there is $\delta>0$ such that for all
$x\in \textcolor{black}{M}$ the set $\Gamma_\delta(x)$ is countable.

In \cite{MoSi} it is proved that the following statements are equivalent:
 \begin{enumerate}
 \item $f$ is countably-expansive,
 \item every non-atomic Borel probability measure of $M$ is expansive with a common expansive
 constant.
  \end{enumerate}
Moreover, \textcolor{black}{they make the following question}: are
there measure-expansive homeomorphisms of compact metric space which
are not countably-expansive? \textcolor{black}{We give a negative
answer in Theorem \ref{count}.}

\textcolor{black}{We next study robust expansiveness of
\mbox{$C^1$-diffeomorphisms} of a smooth manifold.} For a fixed
manifold $M$, we denote by $\mathcal{E}$ the set of all expansive
diffeomorphisms of $M$. \textcolor{black}{In order to state our next
result let us recall more definitions}. We say that $C\subset M$ is
a \emph{continuum} if it is compact and connected. A \emph{trivial
continuum} (or \emph{singleton}) is a continuum with only one point.
Recall from \textcolor{black}{\cites{Ka,K98}} that $f$ is
\emph{continuum-wise expansive} (or \emph{cw-expansive}) if there is
$\delta>0$ such that if $C\subset M$ is a non-trivial continuum then
there is $n\in\Z$ such that $\diam(f^n(C))>\delta$. Denote by
$\mathcal{CE}$ the \textcolor{black}{set} of all cw-expansive
diffeomorphisms and by $\mathcal{PE}$ the set of all
measure-expansive diffeomorphisms of $M$. We denote by $\interior A$
the \mbox{$C^1$-interior} of a set $A$ of
\mbox{$C^1$-diffeomorphisms} of $M$. In \cite{Ma}
R. Mañé proved that the \mbox{$C^1$-interior} of the set of
expansive diffeomorphisms coincides with the set of quasi-Anosov
diffeomorphisms. See \cite{Ma} for the definitions and the proof.
\textcolor{black}{This result was later extended for cw-expansive homeomorphisms in
\cite{Sa97} proving that $\interior{\mathcal{E}}=\interior{\mathcal{CE}}.$
Recently, it was proved in \cite{SaSuYa} that $\interior{\mathcal{E}}=\interior{\mathcal{PE}}.$
In Theorem \ref{maintheorem} we 
give a new proof of the cited result from \cite{SaSuYa} based on Theorem \ref{count} and \cite{Sa97}.}

{\it Acknowledgement.} The authors would like to acknowledge the
many valuable suggestions made by the Professors  H. Miranda and A. Rambaud. \textcolor{black}{The first author thanks to 
Universidad de Bío-Bío for the kind hospitality during the preparation of this work.}

\section{Proofs of the results}
Our first result holds for a homeomorphism $f\colon M\to M$ of a compact metric space $(M,\dist)$.
\begin{thm}
\label{count}
 The following statements are equivalent:
 \begin{enumerate}
 \item $f$ is countably-expansive,
  \item $f$ is measure-expansive.
 \end{enumerate}
\end{thm}

\begin{proof}
\textcolor{black}{\emph{Direct}. Let $\delta>0$ be such that for all
$x\in M$ it holds that $\Gamma_\delta(x)$ is countable. Let $\mu$ be
a non-atomic Borel probability measure. Since $\mu$ is non-atomic,
by $\sigma$-aditivity we have that $\mu(\Gamma_\delta(x))=0$.
Therefore, $f$ is measure-expansive.}

\textcolor{black}{\emph{Converse}.}
 Arguing by contradiction, we assume that $f$ is measure-expansive but there \textcolor{black}{are sequences} $\delta_n\to 0$ and $x_n\in M$
 \textcolor{black}{such that $\Gamma_{\delta_n}(x_n)$ is uncountable for each $n\geq 1$.
 As in \cite{MoSi}, for each $n\geq 1$ consider a non-atomic Borel
 probability measure $\mu_n$ such that
 $\mu_n(\Gamma_{\delta_n}(x_n))=1$.
Consider the Borel probability measure $\mu$ defined for a Borel set $A\subset M$ as
$$\mu(A)=\sum_{n=1}^\infty\frac{\mu_n(A)}{2^n}.$$
Since every $\mu_n$ is non-atomic, we have that $\mu$ is non-atomic too.
Thus, since $f$ is measure-expansive, there is $\delta>0$ such that $\mu(\Gamma_\delta(x))=0$ for all $x\in M$.
Since $\delta_n\to 0$ we can take $\delta_n<\delta$. }
Then
 \[
\mu(\Gamma_{\delta}(x_n))\geq  \mu(\Gamma_{\delta_n}(x_n))\geq \frac{\mu_n(\Gamma_{\delta_n}(x_n))}{2^n}>0.
 \]
 \textcolor{black}{This contradiction proves the theorem}.
\end{proof}

\textcolor{black}{For the proof of Theorem \ref{maintheorem} we
recall some known facts.}

\begin{lem}
The following statements are equivalent:
\begin{enumerate}
\label{equiv}
 \item $f$ is cw-expansive,
 \item there is $\delta>0$ such that for all $x\in M$ it holds that $\Gamma_\delta(x)$ contains no non-trivial continua.
\end{enumerate}
\end{lem}
\begin{proof}
 For the direct part, consider a cw-expansive constant $\epsilon>0$ and take $\delta=\epsilon/2$.
 If $C\subset \Gamma_\delta(x)$ is a connected component then
 $\diam(f^n(C))\leq 2\delta$ for all $n\in\Z$.
 Since $\epsilon=2\delta$ is a cw-expansive constant, we conclude that $C$ is a singleton.
 Then, every continuum contained in $\Gamma_\delta(x)$ is trivial for all $x\in M$.

 In order to prove the converse we consider $\delta>0$ such that every $\Gamma_\delta(x)$ \textcolor{black}{has no} non-trivial \textcolor{black}{continua}.
 Let us show that $\delta$ is a cw-expansive constant.
 Suppose that $C\subset M$ is a continuum and $\diam(f^n(C))\leq \delta$ for all $n\in\Z$.
 Given $x\in C$ we have that for all $y\in C$ it holds that $\dist(f^n(x),f^n(y))\leq\delta$.
 Therefore, $y\in\Gamma_\delta(x)$. Since $y$ is arbitrary, we have that $C\subset \Gamma_\delta(x)$.
 By hypothesis, we have that $\Gamma_\delta(x)$ contains no non-trivial continuum, therefore, $C$ is a singleton.
\end{proof}

\begin{lem}
\label{implic} The following implications hold:
\[
 \hbox{expansive } \Rightarrow \hbox{countably-expansive} \Rightarrow \hbox{cw-expansive}.
\]
\end{lem}

\begin{proof}
 The first implication is obvious because singletons are countable sets.
 The second one holds because every non-trivial continuum is uncountable.
 Therefore, if $\Gamma_\delta(x)$ is countable, it cannot contain any non-trivial continuum.
 \textcolor{black}{By Lemma \ref{equiv} we have that $f$ is cw-expansive}.
\end{proof}

\textcolor{black}{Now assume that $f$ is a
\mbox{$C^1$-diffeomorphism} of a compact smooth manifold $M$ and
recall the definitions from the introduction.}

\begin{thm}[\cite{SaSuYa}]
\label{maintheorem}
 \textcolor{black}{The following equality holds:
 \[
  \interior{\mathcal{E}}=\interior{\mathcal{PE}}.
 \]}
\end{thm}

\begin{proof}
By definition of the sets $\mathcal{E}$, $\mathcal{PE}$ and using
Theorem \ref{count} and Lemma \ref{implic} we have that
\[
\interior{\mathcal{E}}\subset\interior{\mathcal{PE}}\subset\interior{\mathcal{CE}}.
\]
Finally, by Theorem 1 in \cite{Sa97} we have that
$$
\interior{\mathcal{CE}}\subset \interior{\mathcal{E}}.
$$
This finishes the proof.
%
%
\end{proof}

\end{document}